\documentclass[a4paper,12pt]{amsart}
\usepackage{graphicx}
\usepackage{flafter}
\usepackage{amsmath, amsthm, amssymb}
\newtheorem{thm}{Theorem}

\newtheorem{lem}{Lemma}

\theoremstyle{definition}
\newtheorem*{rem}{Remark}

\theoremstyle{definition}
\newtheorem*{qn}{Question}

\newcommand{\eqn}{\begin{equation}}
\newcommand{\eqnend}{\end{equation}}

\newcommand{\Ipc}{\overline{I(p)}^{\,c}}
\newcommand{\Ifc}{\overline{I(f)}^{\,c}}
\newcommand{\Itildefc}{\overline{I(\tilde{f})}^{\,c}}

\newcommand{\Real}{\mathop{\rm Re}\nolimits}
\newcommand{\Imag}{\mathop{\rm Im}\nolimits}

\begin{document}

\title{Wandering domains in quasiregular dynamics}
\author{Daniel A. Nicks}

\begin{abstract}
We show that wandering domains can exist in the Fatou set of a polynomial type quasiregular mapping of the plane. We also give an example of a quasiregular mapping of the plane, with an essential singularity at infinity, which has a sequence of wandering domains contained in a bounded part of the plane. This contrasts with the situation in the analytic case, where wandering domains are impossible for polynomials and, for transcendental entire functions, the existence of wandering domains in a bounded part of the plane has been an open problem for many years.
\end{abstract}

\maketitle

\section{Introduction}

One generalisation of non-linear polynomials in the complex plane is the family of quasiregular mappings of polynomial type for which the degree exceeds the dilatation. The definition of these terms, and a number of others, is deferred to the next section. Here we mention only that a quasiregular map $f:\mathbb{R}^m\to\mathbb{R}^m$ is said to be of \emph{polynomial type} if $f(x)\to\infty$ as $x\to\infty$, and that the \emph{degree} of such a mapping is defined as the maximal number of pre-images of any value.

Suppose that an open set $U$ is completely invariant under a mapping $f$. Let $U_0$ be a component of $U$ and let $U_n$ denote that component of $U$ which contains $f^n(U_0)$. Then we say that $U_0$ is a \emph{wandering component} or a \emph{wandering domain} of $U$ if $U_n\ne U_m$ whenever $n\ne m$.

Motivated by Sullivan's celebrated proof that the Fatou set of a rational function has no wandering components~\cite{Sullivan}, Lasse Rempe asked the following question.

\begin{qn}[{\cite[Problem 6, p.2959--2960]{OWF}}]
Let $p:\mathbb{R}^2\to\mathbb{R}^2$ be a $K$-quasiregular mapping of polynomial type, such that $\deg p >K$. Can the Fatou set of $p$ have a wandering component? 
\end{qn}

Before tackling this question, we must consider what we mean by the Fatou and Julia sets of a quasiregular mapping. 
For a polynomial, it is well known that the Julia set is the boundary of the escaping set. 
We have the following result about the escaping sets of quasiregular mappings of polynomial type. Here, the winding map $re^{i\theta}\mapsto re^{2i\theta}$ shows the sharpness of the degree condition.

\begin{thm}[\cite{FN}]\label{Thm:FN}
Let $p:\mathbb{R}^m\to\mathbb{R}^m$ be a $K$-quasiregular mapping of polynomial type, such that $\deg p> K$. Then the escaping set \[I(p)=\{x\in\mathbb{R}^m:p^n(x)\to\infty \mbox{ as } n\to\infty\}\]
is non-empty, open and connected. The boundary $\partial I(p)$ is a perfect set and is completely invariant under $p$.
\end{thm}

This result suggests that the boundary of the escaping set $\partial I(p)$ may serve as an analogue to the Julia set of a polynomial. In light of this, we could attempt to answer Rempe's question by seeking wandering components of $\Ipc=\mathbb{R}^2\setminus\overline{I(p)}$.

An alternative interpretation of the above question is offered by the work of Sun and Yang \cite{SY1, SY2, SY3} (see also \cite{BergCMFT}), which considers quasiregular mappings of the two-dimensional Riemann sphere. They define the \emph{Julia set} $J(p)$ of a polynomial type quasiregular mapping $p:\mathbb{R}^2\to\mathbb{R}^2$ as the set of points $x\in\mathbb{R}^2$ with the following expanding property: if $U$ is any neighbourhood of $x$, then $\bigcup_np^n(U)$ contains all but at most one point of $\mathbb{R}^2$. Sun and Yang prove that if $\deg p> K(p)$, then the Julia set is non-empty, perfect and completely invariant. It is easy to see that $J(p)\subseteq\partial I(p)$.

For polynomial type mappings $p$ as above, Sun and Yang retain the familiar definition of the \emph{Fatou set}: $x\in F(p)$ if there exists a neighbourhood $U$ of $x$ such that $\{p^n|_U\}$ is a normal family. It should be remarked that under these definitions $\mathbb{R}^2\setminus(J(p)\cup F(p))$ can be non-empty; see for example \cite[Example~5.3]{BergCMFT}.

This note answers Rempe's question by constructing a suitable mapping $p$ which has wandering domains that are components both of $\Ipc$ and of the Fatou set $F(p)$. Our approach is based on the escaping set.

A further interesting feature of this example is that the Julia set is not equal to the boundary of the escaping set, in contrast to the polynomial case. In \cite[Example~5.3]{BergCMFT},  it was already shown that for a polynomial type mapping with $\deg p>K(p)$, the boundary of an attracting basin may not coincide with the Julia set. Note, however, that in our case the attracting fixed point at infinity is even superattracting.

In the final section, we modify the construction to produce a `transcendental' quasiregular mapping of the plane that exhibits wandering in a bounded region. That is, we describe a quasiregular map, with an essential singularity at infinity, which has a wandering domain on which the iterates converge to a finite constant.

\section{Definitions}

Let $G\subseteq\mathbb{R}^2$ be a domain. A continuous function $f:G\to\mathbb{R}^2$ is called \mbox{\emph{$K$-quasiregular}} if it belongs to the Sobolev space $W^{1}_{2, \mathrm{loc}}(G)$ and 
\[ |Df(x)|^2 \le KJ_f(x), \]
for almost every $x\in G$. Here $J_f(x)$ denotes the Jacobian determinant of $f$ at the point $x$, while $|Df(x)|=\sup_{|a|=1}|Df(x)a|$ is the operator norm of the formal derivative of $f$ at $x$. The least $K\ge1$ for which this holds is called the \emph{dilatation} $K(f)$. A function is said to be \emph{quasiregular} if it is $K$-quasiregular for some $K$. 

Any analytic function is $1$-quasiregular, and indeed quasiregular mappings are a natural generalisation of analytic functions. The definition of $K$-quasiregularity in higher dimensions differs slightly from the above and may be found in \cite{BergCMFT, Rickman}, along with more details of many of the concepts defined here.

A non-constant quasiregular mapping $f:\mathbb{R}^2\to\mathbb{R}^2$ is of \emph{polynomial type} if $f(x)\to\infty$ as $x\to\infty$. Otherwise, this limit does not exist and $f$ is said to have an \emph{essential singularity at infinity}. The \emph{degree} of $f$ can be defined by
\[ \deg f := \sup_{y\in\mathbb{R}^2}\left|f^{-1}(y)\right|; \]
that is, the maximal number of pre-images of any value in $\mathbb{R}^2$. It is well known that $\deg f$ is finite if and only if $f$ is of polynomial type. 

An injective quasiregular mapping is called \emph{quasiconformal}. Any quasiregular map $f:\mathbb{R}^2\to\mathbb{R}^2$ has a Stoilow factorisation of the form $f=g\circ h$, where $h$ is quasiconformal and $g$ is analytic \cite[\S3.2]{BergCMFT}. Note that if $f$ is of polynomial type then $g$ is a polynomial.

The composition of two quasiregular maps $f_1$ and $f_2$ is quasiregular with dilatation satisfying $K(f_1\circ f_2)\le K(f_1)K(f_2)$. In particular, the iterates of a quasiregular mapping are also quasiregular.

\section{The main construction}

We answer the question discussed in the introduction by describing a polynomial type quasiregular mapping that has wandering domains.

\subsection{Two quasiconformal maps}\label{2qcs}

The construction takes place in a number of stages, and we shall identify $\mathbb{R}^2$ with $\mathbb{C}$. To begin, let $W_0$ be the open diamond with vertices at $i$, $\frac{i}{2}$, $(3i+1)/4$ and $(3i-1)/4$. See Figure \ref{fig:Wandering}. For $y\in[\frac12,1]$, we define a continuous function $h_y:\mathbb{R}\to\mathbb{R}$ by setting $h_y(x)=1$ when $x+iy\in W_0$, while $h_y(x)=4$ when $|x|\ge y$. For intermediate values of $x$, we ask that $h_y$ is linear in $x$. Explicitly, in the regions where $h_y$ is not locally constant this gives
\[ h_y(x) = \left\{ \begin{array}{ll}
6|x| - 6y +4, & \frac12 \le y \le \frac34, \\
\frac{3|x|+5y-4}{2y-1}, & \frac34 < y \le 1. \raisebox{0mm}[5mm]{}
\end{array} \right. \]
Write $S=\{z:\frac12<\Imag z <1\}$.  We now define a homeomorphism $h:\overline{S}\to\overline{S}$ by 
\eqn h(x+iy) = h_y(x)x + iy, \label{defn h} \eqnend
so that $h$ is locally like a stretch by a factor $h_y(x)$ parallel to the real axis.

We show next that $h$ is quasiconformal on $S$. Note that $h$ is differentiable almost everywhere in $S$ and that $\Imag(h(x+iy))$ is independent of $x$. Hence, the Jacobian determinant of $h$ is given by 
\[ J_h = \frac\partial{\partial x}\left(\Real h(x+iy)\right) \cdot \frac\partial{\partial y}\left(\Imag h(x+iy)\right) = h_y(x)+h_y'(x)x         \quad\mbox{a.e.,}\] 
from which we calculate that $J_h\ge1$ almost everywhere in $S$. 

For $x+iy\in S$ with $|x|>y$, we have that $h(x+iy)=4x+iy$, and hence $|Dh(x+iy)|=4$. By appealing to a compactness argument, there exists $M>0$ such that $|Dh(x+iy)|\le M$ for almost every $x+iy\in\overline{S}$ with $|x|\le y$. Therefore, $h$ is $K$-quasiconformal on $S$ with $K=\max\{4^2,M^2\}$.

\begin{figure}[th]
	\centering
		\includegraphics[width=1.00\textwidth]{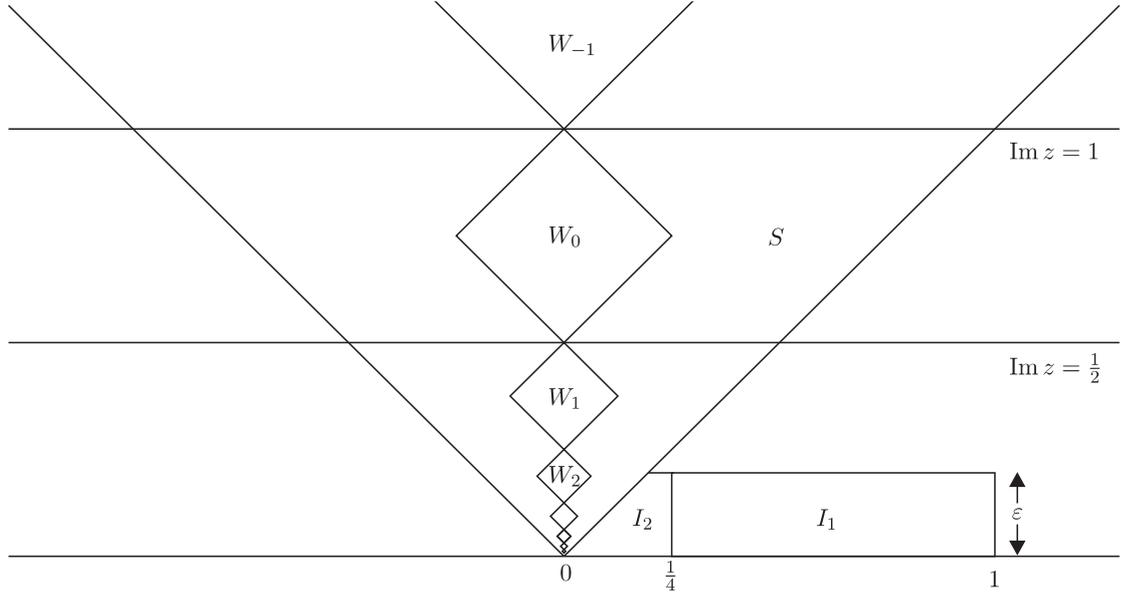}
	\caption{Some of the sets used in the construction.}	
	\label{fig:Wandering}
\end{figure}

We now claim that if $x_0>0$ and $x_0+iy\in\overline{S}\setminus\overline{W_0}$, then
\eqn x_n := \Real(h^n(x_0+iy))\to+\infty \quad \mbox{as } n\to\infty. \label{eqn:x_n to infty} \eqnend
That is, when $h$ is iterated, points to the right of $W_0$ tend to infinity through the right half of $S$.
To prove the claim, note that $h_y(x_0)>1$ and that $x_{n+1}=h_y(x_n)x_n$ by~\eqref{defn h}. Using the fact that $h_y(x)$ is increasing for $x>0$, it follows by induction that $x_n\ge (h_y(x_0))^nx_0$, which establishes \eqref{eqn:x_n to infty}.

The next step is to define, for $z=x+iy\in\mathbb{C}$,
\[ f(z) = \left\{ \begin{array}{ll}
2^{-(m+1)}h(2^mz), & \Imag z \in (2^{-(m+1)},2^{-m}], \ m\in\mathbb{Z},\\
2x + iy/2, & \Imag z \le 0. 
\end{array} \right. \]
Then $f$ is $K$-quasiconformal because in the upper half-plane it is locally the composition of the $K$-quasiconformal map $h$ with two (conformal) scalings.

Let $W_k=2^{-k}W_0$ as shown in Figure \ref{fig:Wandering}. We note that $f(z)=z/2$ on $\bigcup W_k$, since $h$ is the identity map on $W_0$. In fact, $\Imag f(z) = \Imag z/2$ for all $z$ by \eqref{defn h}. Moreover, $f(x+iy)=2x+iy/2$ in the region $\{x+iy:|x|\ge y\}$. These remarks could be used together with the next lemma to give a fairly complete description of the dynamics of $f$.

\begin{lem}\label{lem:re>im}
Let $z$ be such that $\Real z >0$ and $z\notin\bigcup\overline{W_k}$. Then for all sufficiently large $N$, we have 
\eqn \Real(f^N(z))>\Imag(f^N(z)). \label{Re>Im} \eqnend
\end{lem}

\begin{proof}
This is trivial unless $\Imag z>0$, in which case we take $m\in\mathbb{Z}$ such that $\Imag z \in (2^{-(m+1)},2^{-m}]$.
Then $f^N(z)=2^{-(N+m)}h^N(2^mz)$, and so it will suffice to show that, for large $N$,
\eqn  \Real(h^N(2^mz)) > \Imag(h^N(2^mz)) = \Imag(2^mz). \label{eqn:Re(h^N)>Im(h^N)} \eqnend
Since $z\notin\bigcup\overline{W_k}$, we have that $2^mz\in \overline{S}\setminus\overline{W_0}$, and hence \eqref{eqn:x_n to infty} gives \eqref{eqn:Re(h^N)>Im(h^N)} for all large~$N$.
\end{proof}

\subsection{A polynomial type map of high degree with wandering domains}

The mapping $f$ has many of the dynamical properties that we require, but it is an injective function. By composing with another mapping, we now produce a quasiregular map of large degree. Choose an odd integer $d$ such that $d>2K$. Let $\delta>0$ and define $p=g\circ f$, where 
\[ g(z) = \left\{ \begin{array}{ll}
z, & |z|\le 1,\\
z+\delta(|z|-1)z^d, & 1<|z|\le 2,\\
z+\delta z^d, & |z|>2.
\end{array} \right. \]
Hence 
\eqn p(z)=f(z) \quad \mbox{whenever} \quad |f(z)|\le 1. \label{p=f} \eqnend
As in \cite[Examples 5.2--5.4]{BergCMFT}, if we choose $\delta$ small enough, then $g$ is a quasiregular map of degree~$d$ with $K(g)\le2$. Therefore, $p$ is a quasiregular mapping of polynomial type such that
\[ K(p)\le K(g)K(f) \le 2K < d = \deg p. \]
Using \eqref{p=f}, we note that, for $k\ge0$, 
\eqn p\left(\overline{W_k}\right) = \overline{W_{k+1}} \label{p(W_k)=W_k+1}, \eqnend
and so $p^n(z)\to0$ on $\overline{W_k}$ as $n\to\infty$. Moreover, these sets $W_k$ lie in the Fatou set of $p$. Our aim is now to show that the $W_k$ are in fact components of $F(p)$ and of $\Ipc$.

We consider the escaping set $I(p)$. For all real $x>0$, we see that $f(x)=2x$ and $p(x)=g(2x)\ge 2x$, and hence $x\in I(p)$. Then, since $I(p)$ is open by Theorem~\ref{Thm:FN}, there exists $\varepsilon\in(0,\frac14)$ such that 
\[ I_1=\{x+iy : \textstyle\frac14\le x \le 1, \ 0\le y\le\varepsilon\} \subseteq I(p). \]
Let $I_2=\{x+iy: y<x\le\frac14, \ 0\le y\le\varepsilon\}$, as shown in Figure \ref{fig:Wandering}. Recalling \eqref{p=f} and the paragraph preceding Lemma~\ref{lem:re>im}, we find that if $x+iy\in I_2$ then $p(x+iy)=2x+iy/2$. Thus, each point in $I_2$ has some forward iterate under $p$ that lies in $I_1$. This implies that $I_2\subset I(p)$.

Now choose $j\in\mathbb{N}$ such that $2^{-j}\le\varepsilon$ and let $w\in\partial W_j$ with $\Real w\ge0$. Our next task is to show that $w\in\partial I(p)$. For all small $\eta>0$, the point $z=w+\eta$ satisfies the hypothesis of Lemma \ref{lem:re>im}, and so there exists $N\in\mathbb{N}$ such that \eqref{Re>Im} holds but $\Real(f^n(z))\le\Imag(f^n(z))$ for $0\le n <N$. We note that
\[ 0\le \Imag(f^n(z)) = 2^{-n}\Imag z = 2^{-n}\Imag w \le 2^{-n}\varepsilon, \qquad n\ge0. \]
It follows from the definition of $f$ and the above that
\[ 0\le \Real(f^N(z)) \le 2\Real(f^{N-1}(z)) \le 2\Imag(f^{N-1}(z))\le 2\varepsilon. \]
Therefore, we have that $|f^n(z)|<1$ for $n=0,1,\ldots,N$, and using \eqref{p=f} we deduce that $p^n(z)=f^n(z)$ for $n=0,1,\ldots,N$. By \eqref{Re>Im} and the above estimates, it follows that $p^N(z)=f^N(z)\in I_1\cup I_2$. This implies that $w+\eta\in I(p)$ for all small $\eta>0$ which, together with \eqref{p(W_k)=W_k+1}, completes the proof that $w\in\partial I(p)$.

Since the integer $d$ in the definition of $g$ was chosen to be odd, the function $p$ is symmetric about the imaginary axis in the sense that $p\left(-\overline{z}\right)=-\overline{p(z)}$. The escaping set $I(p)$ must share this symmetry, and therefore any point of \mbox{$\partial W_j\cap\{w:\Real w\le0\}$} must also belong to $\partial I(p)$. Hence $\partial W_j\subseteq\partial I(p)$ and, using the complete invariance of $\partial I(p)$ under $p$, we see that in fact $\partial W_k\subseteq\partial I(p)$ for all $k\ge0$. Using~\eqref{p(W_k)=W_k+1}, it now follows that $(W_k)_{k\ge0}$ are wandering components of $\Ipc$.

To show that $(W_k)_{k\ge0}$ are also components of the Fatou set $F(p)$, we need only observe that the family $\{p^n\}$ cannot be normal in any neighbourhood of a point of $\partial W_k$, as these points lie on the boundary of the escaping set.

It remains to prove the assertion made in the introduction that \mbox{$J(p)\ne\partial I(p)$}. To do this, let $j$ be large and consider $z_0\in\partial W_j$ not a vertex of $W_j$. If $U$ is a sufficiently small neighbourhood of $z_0$ then, by an argument similar to that used above, it can be shown that $U\setminus\overline{W_j}\subseteq I(p)$. Recall \eqref{p(W_k)=W_k+1} and the fact that $\overline{W_{j-1}}$ does not meet $I(p)$. Thus, for example, $W_{j-1}\cap\left(\bigcup_n p^n(U)\right)=\emptyset$ and this implies that $z_0\notin J(p)$.

\begin{rem}
A quasiregular mapping is called \emph{uniformly quasiregular} if there exists a uniform bound on the dilatation of all its iterates. 
By a result of Hinkkanen~\cite{Hinkkanen}, every uniformly quasiregular mapping of the plane is quasiconformally conjugate to an entire function. We may therefore deduce from Sullivan's theorem the fact that there are no wandering domains for uniformly quasiregular mappings of the plane of polynomial type.
\end{rem}

\section{A `transcendental' quasiregular map with `bounded wandering'}

A long-standing open question in complex dynamics is whether there exists a transcendental entire function $g$ with a wandering component $W$ of the Fatou set, such that $\bigcup_n g^n(W)$ is a bounded set (see for example \cite[Question~8]{Berg_Mero} or \cite[Problem~2.87]{BH}). In this section, we show that there does exist a quasiregular mapping of the plane, with an essential singularity at infinity, that has this `bounded wandering' property.

The quasiconformal map $f$ defined in Section \ref{2qcs} has a sequence of wandering domains $W_k$ that converge to the origin. (Lemma~\ref{lem:re>im} and the paragraph preceding it imply that the $W_k$ are wandering components of both $F(f)$ and $\Ifc$.) In fact, this remains true when $f$ is considered as a self-map of the upper half-plane. Starting from $f$, we construct a quasiregular mapping $\tilde{f}$ with the desired properties by using an interpolation technique from \cite[\S 6]{BFLM} to insert an essential singularity at infinity over the lower half-plane.

Let $\delta>0$ be small and define $\tilde{f}:\mathbb{C}\to\mathbb{C}$ by
\[ \tilde{f}(z) = \left\{ \begin{array}{ll}
f(z), & \Imag z \ge 0,\\
2z - \delta(\Imag z)\exp(-z^2), & -1\le\Imag z < 0,\\
2z + \delta\exp(-z^2), & \Imag z <-1.
\end{array} \right. \]
It is clear that $\tilde{f}$ is continuous on $\mathbb{C}$, quasiregular on the upper half-plane and analytic on $\{z:\Imag z<-1\}$. We show next that $\tilde{f}$ is quasiregular on 
\[ D=\{z:-1<\Imag z<0\}. \] 
Write $\phi(z)=(\Imag z)\exp(-z^2)$ and let $x$ and $y$ denote the real and imaginary parts of $z$ respectively. Then $\phi$ is continuous on $D$ and the partial derivatives $\phi_x$ and $\phi_y$ are bounded on $D$ because $\exp(-z^2)$ tends to zero rapidly as $z\to\infty$ in $D$. If $\delta$ is chosen sufficiently small, then it follows that $\tilde{f}(z)=2z-\delta\phi(z)$ is quasiregular on $D$; see \cite[p.647]{BFLM}. Here we have used the fact that the function $z\mapsto 2z$ is analytic with derivative bounded away from zero. Therefore, $\tilde{f}$ is quasiregular on the plane.

The exponential rate of growth of $\tilde{f}(-ir)\sim\delta\exp(r^2)$ as real $r\to\infty$ shows that $\tilde{f}$ is not of polynomial type; see \cite[Theorem~4.1]{Jarvi} for example.

Using the fact that $\tilde{f}$ coincides with $f$ on the upper half-plane, and arguing as before, we conclude that the sets $W_k$ (as defined in Section~\ref{2qcs}) are wandering components of $\Itildefc$ and of $F(\tilde{f})$. In particular, $\tilde{f}$ maps $W_k$ onto $W_{k+1}$, and hence in any fixed $W_k$ the sequence of iterates $\tilde{f}^n$ uniformly converges to zero as $n$ tends to infinity.

\begin{rem}
An important feature of the above example is that the fixed point of $\tilde{f}$ at the origin behaves like a saddle point. This type of behaviour is clearly not possible for an entire function. Moreover, P\'{e}rez-Marco \cite{P-M} has proved that, for an entire function, any orbit which converges to an irrationally indifferent fixed point is eventually constant. Thus the Fatou set of an entire function cannot have a wandering domain on which the iterates converge to a fixed point.
\end{rem}

The author would like to thank Walter Bergweiler and Alastair Fletcher for useful discussions, and the former for bringing \cite{P-M} to his attention. The author also thanks Phil Rippon and Gwyneth Stallard for asking about the existence of an example of bounded wandering with an essential singularity at infinity.

\end{document}